\documentclass[11pt]{article}
\usepackage{graphicx}
\usepackage{tikz} 
\usepackage{pgfplots} 
\usepackage{float} 
\usepackage{indentfirst} 
\usepackage{todonotes} 
\usepackage{amsthm}
\usepackage{amssymb}
\usepackage[all]{xy}
\usepackage{eepic}
\usepackage{multirow}
\usepackage{siunitx}
\usepackage{color}

\newcommand{\typeA}{\ensuremath{\alpha}} 
\newcommand{\typeB}{\ensuremath{\beta}}

\providecommand{\keywords}[1]{\textbf{Keywords:} #1}

\newcommand{\ignore}[1]{}
\newcommand{\mytodo}[1]{}

\newcommand{\poset}{\ensuremath{P}}
\newcommand{\perm}{\ensuremath{v}}
\newcommand{\pattern}{\ensuremath{w}}
\newcommand{\otherpattern}{\ensuremath{y}}
\newcommand{\someelt}{\ensuremath{e}}
\newcommand{\elt}[2]{\ensuremath{\someelt_{#1,#2}}}
\newcommand{\otherelt}{\ensuremath{g}}

\newcommand{\toothlen}{\ensuremath{t}}
\newcommand{\tootheltidx}{\ensuremath{\tau}}
\newcommand{\numteeth}{\ensuremath{s}}
\newcommand{\spineeltidx}{\ensuremath{\sigma}}
\newcommand{\numelts}{\ensuremath{n}}

\newcommand{\plaincomb}{\ensuremath{K}}
\newcommand{\comb}[2]{\ensuremath{\plaincomb_{{#1},{#2}}}}
\newcommand{\combA}[2]{\ensuremath{\comb{#1}{#2}^{\typeA}}}
\newcommand{\combB}[2]{\ensuremath{\comb{#1}{#2}^{\typeB}}}
\newcommand{\tmcomb}{\comb{\numteeth}{\toothlen}}
\newcommand{\plainunevencomb}{\ensuremath{U}}
\newcommand{\unevencomb}[2]{\ensuremath{\plainunevencomb_{#1, #2}}}
\newcommand{\unevencombA}[2]{\ensuremath{\unevencomb{\mathsf{numteeth} = #1}{#2}^{\typeA}}}
\newcommand{\unevencombB}[2]{\ensuremath{\unevencomb{\mathsf{toothlen} = #1}{#2}^{\typeB}}}

\newcommand{\parentnode}{\ensuremath{r}}

\newcommand{\someint}{\ensuremath{c}}
\newcommand{\someotherint}{\ensuremath{d}}

\newcommand{\children}{\ensuremath{C}}
\newcommand{\numchildren}{\ensuremath{k}}

\newcommand{\numexts}[1]{\ensuremath{E(#1)}}
\newcommand{\numavexts}[2]{\ensuremath{A_{#1}(#2)}}

\newcommand{\observedval}[1]{#1}
\newcommand{\extrapval}[1]{#1}
\newcommand{\suggestedval}[1]{\textbf{ #1}}
\newcommand{\proventable}[1]{}

\begin{document}

\newcommand{\YourTitle}{Pattern Avoidance in Extensions of Comb-Like Posets}
\newcommand{\YourRunningTitle}{\YourTitle}
\newcommand{\FirstAuthor}{Sophia Yakoubov}
\newcommand{\AuthorsEmails}{\texttt{sonka89@mit.edu}}
\newcommand{\AuthorsAddress}{MIT}

\title{
\includegraphics[scale=0.25]{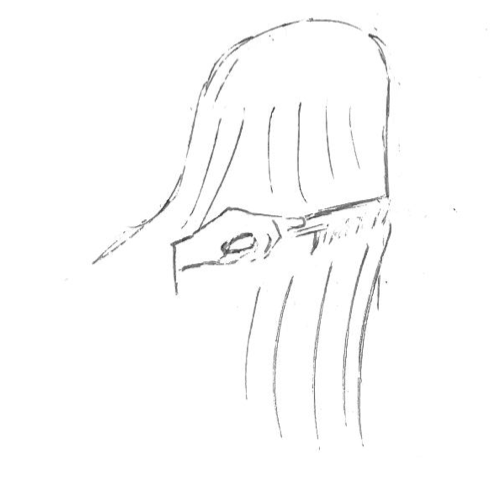} \\
\YourTitle
}
\author{
\FirstAuthor}
\date{
\AuthorsEmails\\\vspace{6pt}
\AuthorsAddress
} 

\maketitle

\newtheorem{thm}{Theorem}
\newtheorem{conjecture}{Conjecture}
\theoremstyle{definition} \newtheorem{definition}{Definition} 

\newenvironment{customlegend}[1][]{%
    \begingroup
    \csname pgfplots@init@cleared@structures\endcsname
    \pgfplotsset{#1}%
}{%
    \csname pgfplots@createlegend\endcsname
    \endgroup
}%

\def\addlegendimage{\csname pgfplots@addlegendimage\endcsname}

\begin{abstract}

This paper investigates pattern avoidance in linear extensions of particular partially ordered sets (posets). Since the problem of enumerating pattern-avoiding linear extensions of  posets without any additional restrictions is a very hard one, we focus on the class of posets called \emph{combs}. 
A comb consists of a fully ordered \emph{spine} and several fully ordered \emph{teeth}, where the first element of each tooth coincides with a corresponding element of the spine.
We consider two natural assignments of integers to elements of the combs; we refer to the resulting integer posets as \emph{type-{\typeA}} and \emph{type-{\typeB}} combs. 
In this paper, we enumerate the linear extensions of type-{\typeA} and type-{\typeB} combs that avoid some of the length-3 patterns $\pattern \in S_3$. Most notably, we shown the number of linear extensions of type-{\typeB} combs that avoid $312$ to be the same as the number $\frac{1}{\numteeth \toothlen + 1}{{\numteeth(\toothlen+1)}\choose{\numteeth}}$ of $(\toothlen+1)$-ary trees on {\numteeth} nodes, where {\toothlen} is the length of each tooth and {\numteeth} is the length of the comb spine or, equivalently, the number of its teeth. We also investigate the enumeration of linear extensions of type-{\typeA} and type-{\typeB} combs that avoid multiple length-3 patterns simultaneously.

\end{abstract}


\section{Introduction}

This work introduces a new family of enumeration problems.
We consider enumerating permutations with two separate constraints. 
The first of these is that the permutation avoids a given pattern, as described in Section \ref{sec:pa}.
The second of these is that the relative ordering of certain elements is fixed, i.e. the permutation is a \emph{linear extension of a partially ordered set (poset)}, as described in Section \ref{sec:posets}.

This is an incredibly general problem. 
In particular, if we set the poset under consideration to be the trivial poset with no minimal relations (whose linear extensions include all permutations of $\{1, 2, \ldots, \numelts\}$),
then the problem becomes that of enumerating all integer permutations avoiding certain patterns. 
The topic of pattern avoidance in integer permutations 
is a very active area of exploration in combinatorics. 
Enumerating permutations avoiding certain patterns (most notably the patterns $1324$ and $4231$) has been a longstanding and much-worked-upon open problem.

However, we chose to focus not on the trivial poset, but on a different category of posets called \emph{combs}, described in Sections \ref{sec:combs} and \ref{sec:abcombs}. 
Characterizing the subset of comb linear extensions avoiding certain patterns is a much more approachable task than doing so for general posets, and it is a good first step in the direction of studying this space. 

Section \ref{sec:background} provides some background on the topics of pattern avoidance and partially ordered sets.
Sections \ref{sec:combs} and \ref{sec:abcombs} define combs, and describe the two different assignments of integers to comb elements which we consider.
Sections \ref{sec:numavexts312combB}, \ref{sec:avalphaone}, \ref{sec:avbetaone}, \ref{sec:avalphatwo} and \ref{sec:avbetatwo} detail our results, which deal with the avoidance of all possible length-3 patterns in comb linear extensions, as well as with the avoidance of two such patterns at a time.
\section{Background}
\label{sec:background}

This paper explores the enumeration of a certain type of permutation called a \emph{pattern-avoiding comb linear extension}.
We build on several popular combinatorial concepts, such as \emph{pattern avoidance} and \emph{posets}, which will be introduced briefly in this section. 
If these concepts are already familiar to you, then you can safely skip Sections \ref{sec:pa} and \ref{sec:posets}. 
However, \emph{combs}, which are described in Sections \ref{sec:combs} and \ref{sec:abcombs}, are fairly specific to this paper.

\subsection{Pattern avoidance} 
\label{sec:pa}

Let {\pattern} be a permutation (or ordering) of the integers $\{1, \ldots, m\}$. The set of all permutations of $\{1, \ldots, m\}$ is denoted $S_m$, so we can write $\pattern \in S_m$. 
Another permutation $\perm \in S_n$ (where $n \geq m$) is said to \emph{contain} {\pattern} if there is a subsequence of {\perm} that has the same relative order as {\pattern}; 
it is said to be \emph{{\pattern}-avoiding} if no subsequence of {\perm} has the same relative order as {\pattern}. 
Note that the subsequences we consider do not necessarily have to be consecutive.

For instance, let $\perm = 31524$. {\perm} avoids $321$ but contains $123$, since $124$ is a subsequence of {\perm} has the same relative order as $123$. 

MacMahon \cite{MacMahon} postulated and Knuth and Rotem \cite{Knuth, Rotem} proved that for any $\pattern \in S_3$, there are $C_n$ {\pattern}-avoiding permutations $\perm \in S_n$, where $C_n = \frac{1}{n+1}{{2n}\choose{n}}$ is the $n$th Catalan number. 
In general, pattern avoidance is a very active area of combinatorics. It is discussed in depth in Chapters 4 and 5 of \textit{Combinatorics of Permutations} by Miklos B\'ona \cite{Bona}.

\subsection{Posets} 
\label{sec:posets}

Let {\poset} be a finite \emph{partially ordered set}, or \emph{poset}. 
For poset notation and terminology, we follow Chapter 3 of Enumerative Combinatorics, Volume 1 by Richard Stanley \cite{EC1}.
A \emph{linear extension} of poset {\poset} is a permutation of {\poset}'s elements that respects {\poset}'s relations; so, if {\poset} has the relation $a \leq_{\poset} b$, then in any linear extension of {\poset}, $a$ will precede $b$.
To illustrate, if {\poset} has the Hasse diagram in Figure \ref{fig:exampleposet}, then the linear extensions of {\poset} are $[\elt{1}{1}, \elt{1}{2}, \elt{2}{1}, \elt{2}{2}]$ and $[\elt{1}{1}, \elt{2}{1}, \elt{1}{2}, \elt{2}{2}]$.

\begin{figure} [h!]
\begin{displaymath}
    \xymatrix@!=2ex{
        & \elt{2}{2} & \\
        \elt{1}{2} \ar@{-}[ur] & & \elt{2}{1} \ar@{-}[ul] \\
        & \elt{1}{1} \ar@{-}[ur] \ar@{-}[ul] & }
\end{displaymath}
\caption{Example of a Hasse diagram of a poset}
\label{fig:exampleposet}
\end{figure}
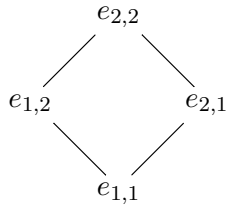

Let {\numavexts{\pattern}{\poset}} represent the number of linear extensions of poset {\poset} which avoid {\pattern}, as described in Section \ref{sec:pa}.

\section{Combs} \label{sec:combs}

This paper focuses on a specific class of posets called \textit{combs} (Definition \ref{def:comb}). Each \emph{comb} has a \emph{spine}, which is a fully ordered set of {\numteeth} elements. Each element of the spine has a corresponding \emph{tooth} of size {\toothlen}, which is also fully ordered, and the first element of which coincides with the corresponding spinal element.

\begin{figure} [h!] 
$$\xymatrix@!=2ex{
        & & & \elt{4}{3} & & \\
        & & \elt{3}{3} & & \elt{4}{2} \ar@{-}[ul] & \\
        & \elt{2}{3} & & \elt{3}{2} \ar@{-}[ul] & & \elt{4}{1} \ar@{-}[ul] \\
        \elt{1}{3} & & \elt{2}{2} \ar@{-}[ul] & & \elt{3}{1} \ar@{-}[ul] \ar@{-}[ur] & \\
        & \elt{1}{2} \ar@{-}[ul] & & \elt{2}{1}  \ar@{-}[ur] \ar@{-}[ul] & & \\
        & &  \elt{1}{1} \ar@{-}[ur] \ar@{-}[ul] & & & }
$$
\caption{Hasse diagram of a comb with spine of length $\numteeth = 4$ and teeth of length $\toothlen = 3$}
\label{fig:comb1}
\end{figure}
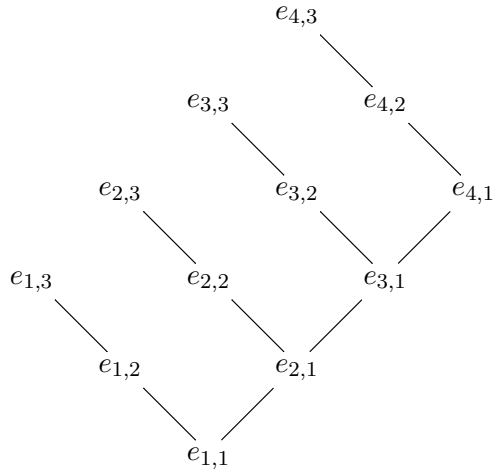

In Figure \ref{fig:comb1}, elements $\elt{1}{1}$, $\elt{2}{1}$ and $\elt{3}{1}$ form the spine of the comb. 
The first tooth consists of elements $\elt{1}{1}$, $\elt{1}{2}$ and $\elt{1}{3}$, 
the second tooth consists of elements $\elt{2}{1}$, $\elt{2}{2}$ and $\elt{2}{3}$, 
and the third tooth consists of elements $\elt{3}{1}$, $\elt{3}{2}$ and $\elt{3}{3}$.

\begin{definition} \label{def:comb}
More formally, let the \emph{comb} with \emph{spine} of length {\numteeth} and \emph{teeth} of length {\toothlen}, denoted {\tmcomb}, be a poset with elements $\{\elt{\spineeltidx}{\tootheltidx}\}_{\spineeltidx \in [\numteeth], \tootheltidx \in [\toothlen]}$ such that 
\begin{itemize}
\item $\elt{1}{1} \leq_{\plaincomb} \elt{2}{1} \leq_{\plaincomb} \cdots \leq_{\plaincomb} \elt{\numteeth}{1}$, and 
\item $\elt{\spineeltidx}{1} \leq_{\plaincomb} \elt{\spineeltidx}{2} \leq_{\plaincomb} \cdots \leq_{\plaincomb} \elt{\spineeltidx}{\toothlen}$ for $1 \leq \spineeltidx \leq \numteeth$.
\end{itemize}
The spine of the comb consists of elements $\elt{1}{1}, \ldots, \elt{\numteeth}{1}$, 
and the teeth consist of elements $\elt{\spineeltidx}{1}, \ldots, \elt{\spineeltidx}{\toothlen}$ for each integer {\spineeltidx} such that $1 \leq \spineeltidx \leq \numteeth$.

\end{definition}

\subsection{Enumerating the linear extensions of comb {\tmcomb}}
Combs are a special case of posets the Hasse diagrams of which form rooted trees.
Theorem \ref{thm:knuthextenum}, stated by Donald E. Knuth \cite{Knuth}, enumerates the linear extensions of all such posets.

\begin{definition}
Let a \emph{descendant} of element {\someelt} in poset {\poset} be any element {\someelt'} which is forced to come no earlier than {\someelt} by the minimal relations of {\poset}. (Note that the descendants of {\someelt} necessarily include {\someelt} itself.)
\end{definition}

\begin{thm}
\label{thm:knuthextenum}
Let {\poset} be a poset with {\numelts} elements, the Hasse diagram of which forms a rooted tree such that the $i$th element has $d_i$ descendants. {\poset} then has $\numexts{\poset}$ linear extensions, where $\numexts{\poset}$ is defined as follows:
\[\numexts{\poset} = \frac{\numelts!}{\prod_{i \in [\numelts]} d_i}.\]
\end{thm}

It is easy to see that the number of linear extensions of the comb {\tmcomb} is then
\[\numexts{\tmcomb} = \frac{(\numteeth \toothlen)!}{\numteeth!(\toothlen!)^{\numteeth}}.\]

\section{Type-{\typeA} and type-{\typeB} combs} 
\label{sec:abcombs}

For the rest of this paper, we focus on two specific families of combs with integer elements, both of which always have the identity permutation $[1, 2, \ldots, \toothlen \numteeth]$ as an extension.
In type-{\typeA} combs, we set $\elt{\spineeltidx}{\tootheltidx} = (\tootheltidx - 1)\numteeth + \spineeltidx$; that is, the spine holds the first {\numteeth} integers. In type-{\typeB} combs, we set $\elt{\spineeltidx}{\tootheltidx} = (\spineeltidx - 1)\toothlen + \tootheltidx$; that is, the spine holds all integers of the form $\toothlen \someint+1$ (for integer values of {\someint}) in increasing order. More formally, 

\begin{definition}
A \emph{type-{\typeA} comb} {\combA{\numteeth}{\toothlen}} with {\numteeth} teeth of length {\toothlen} has the following properties:
\begin{itemize}
\item $1 \leq_{\plaincomb} 2 \leq_{\plaincomb} \cdots \leq_{\plaincomb} \numteeth$, and 
\item $\someint \leq_{\plaincomb} \someint+\numteeth \leq_{\plaincomb} \someint+2\numteeth \leq_{\plaincomb} \cdots \leq_{\plaincomb} \someint+(\toothlen-1)\numteeth$, 
for integers {\someint} such that $1 \leq \someint \leq \numteeth$.
\end{itemize} 
\end{definition}

Several type-{\typeA} combs are shown in Figure \ref{fig:comb3}.

\begin{figure}[h!] 
$$\begin{array}{p{1.5in}p{1.5in}p{1.5in}}
\xymatrix{
       2 & \\
       & 1 \ar@{-}[ul] }
&
\xymatrix{
        & 4 & \\
        3 & & 2 \ar@{-}[ul] \\
        & 1 \ar@{-}[ur] \ar@{-}[ul] & }
&
\xymatrix{
        & & 6 & \\
        & 5 & & 3 \ar@{-}[ul] \\
        4 & & 2 \ar@{-}[ur] \ar@{-}[ul] & \\
        & 1 \ar@{-}[ur] \ar@{-}[ul] & & }\\
& & \\
\center{\combA{1}{2}} & \center{\combA{2}{2}} & \center{\combA{3}{2}}
\end{array}$$
\caption{Hasse diagram of the first few type-{\typeA} combs with teeth of length $\toothlen = 2$}
\label{fig:comb3}
\end{figure}

\begin{definition}
A \emph{type-{\typeB} comb} {\combB{\numteeth}{\toothlen}} with {\numteeth} teeth of length {\toothlen} has the following properties:
\begin{itemize}
\item $1 \leq_{\plaincomb} \toothlen+1 \leq_{\plaincomb} 2\toothlen+1 \leq_{\plaincomb} \cdots \leq_{\plaincomb} (\numteeth-1)\toothlen+1$, and 
\item $\someint \toothlen + 1 \leq_{\plaincomb} \someint \toothlen + 2 \leq_{\plaincomb} \cdots \leq_{\plaincomb} \someint \toothlen + \toothlen$,
for integers {\someint} such that $0 \leq \someint \leq \numteeth - 1$.
\end{itemize} 
\end{definition}

Several type-{\typeB} combs are shown in Figure \ref{fig:comb4}.

\begin{figure}[h!] 
$$\begin{array}{p{1.5in}p{1.5in}p{1.5in}}
\xymatrix{
       2 & \\
       & 1 \ar@{-}[ul] }
&
\xymatrix{
        & 4 & \\
        2 & & 3 \ar@{-}[ul] \\
        & 1 \ar@{-}[ur] \ar@{-}[ul] & }
&
\xymatrix{
        & & 6 & \\
        & 4 & & 5 \ar@{-}[ul] \\
        2 & & 3 \ar@{-}[ur] \ar@{-}[ul] & \\
        & 1 \ar@{-}[ur] \ar@{-}[ul] & & }\\
& & \\
\center{\combB{1}{2}} & \center{\combB{2}{2}} & \center{\combB{3}{2}}
\end{array}$$
\caption{Hasse diagram of the first few type-{\typeB} combs with teeth of length $\toothlen = 2$}
\label{fig:comb4}
\end{figure}
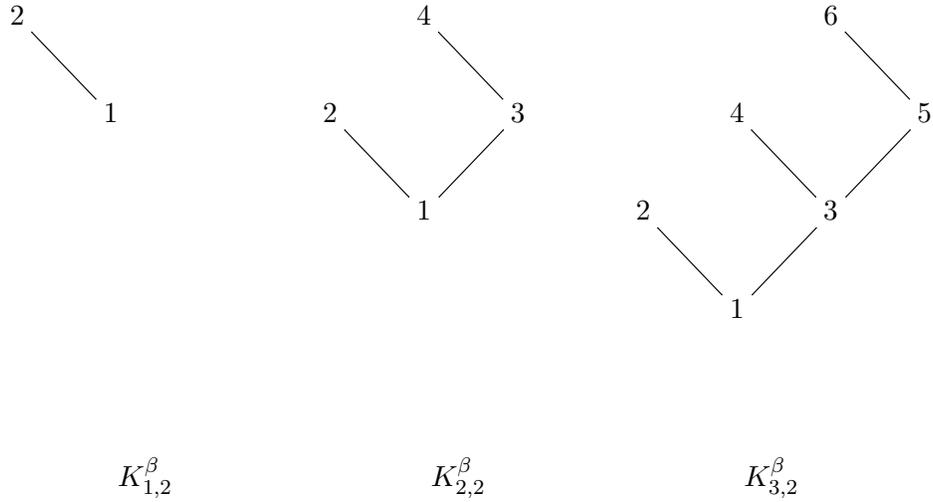

\subsection{Uneven combs} 
\label{sec:unevencombs}

Throughout this paper, we will need to describe the insertion of the next (the $[\numelts + 1]$st) element into a linear extension of an $\numelts$-element comb.
To do this, we will need to refer to combs whose number of elements may not be a multiple of {\toothlen}. We introduce some new notation for such \emph{uneven combs} here.

In type-{\typeA} combs, the next element is always appended to the end of an existing tooth. Let $\unevencombA{\numteeth}{\numelts}$ denote the type-{\typeA} uneven comb with {\numteeth} teeth and a total of {\numelts} elements (Figure \ref{fig:unevencombs}).
If {\numelts} is not divisible by {\numteeth}, then the last several teeth of {\unevencombA{\numteeth}{\numelts}} will contain one fewer element than the rest.
Note that $\unevencombA{\numteeth}{\toothlen \numteeth} = \combA{\numteeth}{\toothlen}$.

\begin{figure} [H] 
$$\xymatrix@!=2ex{
        & & & & 8 &  \\
        & 10 & & 7 & & 4 \ar@{-}[ul]  \\
        9 & & 6 \ar@{-}[ul] & & 3  \ar@{-}[ur] \ar@{-}[ul] & \\
        & 5 \ar@{-}[ul] & &  2 \ar@{-}[ur] \ar@{-}[ul] & & \\
        & & 1 \ar@{-}[ur] \ar@{-}[ul] & & & }
$$
\caption{Hasse diagram of $\unevencombA{4}{10}$, with two teeth of length $\toothlen = 3$ and two teeth of length $\toothlen = 2$}
\label{fig:unevencombs}
\end{figure}

In type-{\typeB} combs, the next element is always appended to the last tooth (or to the spine, if the last tooth is already of length {\toothlen}). Let $\unevencombB{\toothlen}{\numelts}$ denote the type-{\typeB} uneven comb with teeth of length {\toothlen} and a total of {\numelts} elements (Figure \ref{fig:unevencombt}). 
If {\numelts} is not divisible by {\toothlen}, then {\unevencombB{\toothlen}{\numelts}} will have an additional, shorter tooth at the end of the spine.
Note that $\unevencombB{\toothlen}{\toothlen \numteeth} = \combB{\numteeth}{\toothlen}$.

\begin{figure} [H] 
$$\xymatrix@!=2ex{
        & & 9 & & &  \\
        & 6 & & 8 \ar@{-}[ul] & & 10  \\
        3 & & 5 \ar@{-}[ul] & & 7  \ar@{-}[ur] \ar@{-}[ul] & \\
        & 2 \ar@{-}[ul] & &  4 \ar@{-}[ur] \ar@{-}[ul] & & \\
        & & 1 \ar@{-}[ur] \ar@{-}[ul] & & & }
$$
\caption{Hasse diagram of $\unevencombB{3}{11}$, with three teeth of length $\toothlen = 3$ and an additional tooth of length $\toothlen = 1$}
\label{fig:unevencombt}
\end{figure}

The remainder of this paper attempts to enumerate the {\pattern}-avoiding linear extensions of {\combA{\numteeth}{\toothlen}} and {\combB{\numteeth}{\toothlen}} for $\pattern \in S_3$.
We begin with the most interesting case, with the most complex proof: $312$-avoiding linear extensions of type-{\typeB} combs.

\section{$312$-avoidance in type-{\typeB} combs}
\label{sec:numavexts312combB}

In this section, we will enumerate the $312$-avoiding linear extensions of {\combB{\numteeth}{\toothlen}}.
More specifically, we will prove that the number of $312$-avoiding linear extensions of {\combB{\numteeth}{\toothlen}} is the same as the number of $(\toothlen + 1)$-ary trees on {\numteeth} nodes.

\begin{definition}
A $(\toothlen + 1)$-ary tree on {\numteeth} nodes is a rooted tree in which each node has at most $\toothlen + 1$ children. 
The nodes are unlabeled and only the shape of the tree is considered, as illustrated in Figures \ref{fig:threenodebinarytrees} and \ref{fig:threenodeternarytrees}. 
Specifically, note that the position of each node's children is significant (e.g. right vs. left in the binary trees of Figure \ref{fig:threenodebinarytrees}, and right vs. center vs. left in the ternary trees of Figure \ref{fig:threenodeternarytrees}).
\end{definition}

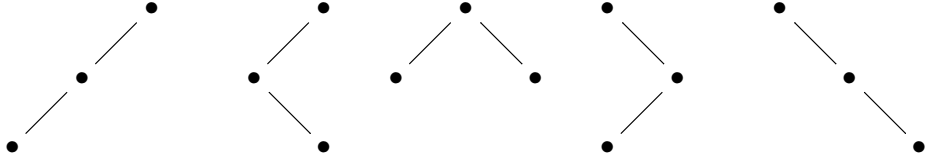
\begin{figure}[H]
\begin{center}
$\xymatrix@!=.5ex{
        & & \bullet \ar@{-}[dl] \\
        & \bullet \ar@{-}[dl] & \\
        \bullet & &}$
$\xymatrix@!=.5ex{
        & & \bullet \ar@{-}[dl] \\
        & \bullet \ar@{-}[dr] & \\
        & & \bullet}$
\hspace{3 mm}
$\xymatrix@!=.5ex{
        & \bullet \ar@{-}[dl]  \ar@{-}[dr] & \\
        \bullet & & \bullet \\
        & &}$
\hspace{3 mm}
$\xymatrix@!=.5ex{
        \bullet \ar@{-}[dr] & & \\
        & \bullet \ar@{-}[dl] & \\
        \bullet & &}$
$\xymatrix@!=.5ex{
        \bullet \ar@{-}[dr] & & \\
        & \bullet \ar@{-}[dr] & \\
        & & \bullet }$
\end{center}
\caption{All five binary trees on three nodes}
\label{fig:threenodebinarytrees}
\end{figure}

\newcommand{\thisvspace}{\vspace{7 mm}}

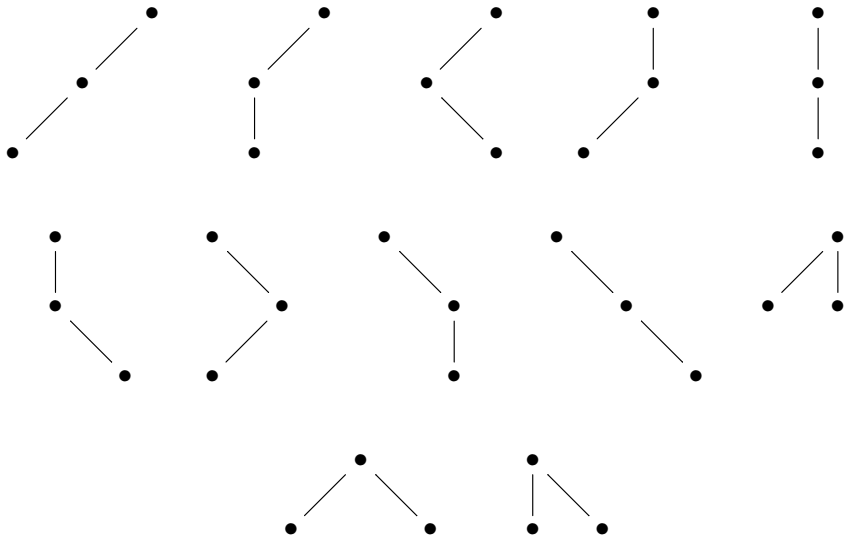
\begin{figure}[H]
\begin{center}
$\xymatrix@!=.5ex{
        & & \bullet \ar@{-}[dl]\\
        & \bullet \ar@{-}[dl] & \\
        \bullet & &}$
$\xymatrix@!=.5ex{
        & & \bullet \ar@{-}[dl]\\
        & \bullet \ar@{-}[d] & \\
        & \bullet &}$
$\xymatrix@!=.5ex{
        & & \bullet \ar@{-}[dl]\\
        & \bullet \ar@{-}[dr] & \\
        & & \bullet}$
\hspace{5 mm}
$\xymatrix@!=.5ex{
        & \bullet \ar@{-}[d] & \\
        & \bullet \ar@{-}[dl] & \\
        \bullet & &}$
\thisvspace
$\xymatrix@!=.5ex{
        & \bullet \ar@{-}[d] & \\
        & \bullet \ar@{-}[d] & \\
        & \bullet &}$
$\xymatrix@!=.5ex{
        & \bullet \ar@{-}[d]& \\
        & \bullet \ar@{-}[dr] & \\
        & & \bullet}$
 \hspace{5 mm}
$\xymatrix@!=.5ex{
        \bullet \ar@{-}[dr] & & \\
        & \bullet \ar@{-}[dl] & \\
        \bullet & &}$
$\xymatrix@!=.5ex{
        \bullet \ar@{-}[dr] & & \\
        & \bullet \ar@{-}[d] & \\
        & \bullet &}$
$\xymatrix@!=.5ex{
        \bullet \ar@{-}[dr] & & \\
        & \bullet \ar@{-}[dr] & \\
        & & \bullet}$
\hspace{3 mm}
\thisvspace
$\xymatrix@!=.5ex{
        & \bullet \ar@{-}[dl] \ar@{-}[d] & \\
        \bullet & \bullet & }$
$\xymatrix@!=.5ex{
        & \bullet \ar@{-}[dl] \ar@{-}[dr] & \\
        \bullet & & \bullet }$
$\xymatrix@!=.5ex{
        & \bullet \ar@{-}[d] \ar@{-}[dr] & \\
        & \bullet & \bullet }$
\end{center}
\caption{All twelve ternary trees on three nodes}
\label{fig:threenodeternarytrees}
\end{figure}

The number of $(\toothlen + 1)$-ary trees on {\numteeth} nodes is known to be $\frac{1}{\toothlen \numteeth+1}{\numteeth(\toothlen+1) \choose \numteeth}$  \cite{Aval}.

\begin{thm} 
\label{thm:main}
For $\toothlen > 1$, 
$\numavexts{312}{\combB{\numteeth}{\toothlen}} = \frac{1}{\toothlen \numteeth+1}{\numteeth(\toothlen+1) \choose \numteeth}$.
\end{thm}

\proventable{
The following table shows in bold the observed number {\numavexts{312}{\combB{\numteeth}{\toothlen}}} for various values of $\toothlen > 1$ and {\numteeth};
the other values have been filled in based on Theorem \ref{thm:main}.
\begin{figure}[H]
\begin{center}
\begin{tabular}{ c | c | c | c }
  \hline
  {\numteeth} 
  &  \numavexts{312}{\combB{\numteeth}{2}}
  &  \numavexts{312}{\combB{\numteeth}{3}} 
  &  \numavexts{312}{\combB{\numteeth}{4}} \\ \hline
1 & \observedval{1} & \observedval{1} & \observedval{1} \\  
2 & \observedval{3} & \observedval{4} & \observedval{5} \\ 
3 & \observedval{12} & \observedval{22} & \extrapval{35} \\ 
4 & \observedval{55} & \observedval{140} & \extrapval{285}  \\ 
5 & \observedval{273} & \extrapval{969} & \extrapval{2530} \\ 
6 & \observedval{1428} & \extrapval{7084} & \extrapval{23751} \\ 
  \end{tabular}
\end{center}
\caption{\numavexts{312}{\combB{\numteeth}{\toothlen}} for various values of {\toothlen} and {\numteeth}}
\end{figure}

Notice that the observed sequence for {\numavexts{312}{\combB{\numteeth}{2}}} matches $\frac{1}{2\numteeth+1}{{3\numteeth}\choose{\numteeth}}$, the observed sequence for {\numavexts{312}{\combB{\numteeth}{3}}} matches $\frac{1}{3\numteeth+1}{{4\numteeth}\choose{\numteeth}}$, and the observed sequence for {\numavexts{312}{\combB{\numteeth}{4}}} matches $\frac{1}{4\numteeth+1}{{5\numteeth}\choose{\numteeth}}$.
}

\subsection{Proof of Theorem \ref{thm:main}}
\label{sec:mainproof}

\begin{proof}

In order to prove Theorem \ref{thm:main}, we will use \emph{generating tree analysis}, 
which is the process of figuring out how many objects of size $\numelts+1$ can be derived from each object of size {\numelts}.

A \emph{generating tree} is a tree each node of which represents an instance of an object. 
The objects represented by the {\numelts}th level of the tree are of size {\numelts}. 
The children of node {\parentnode}, where {\parentnode} is on the {\numelts}th level of the tree, represent the objects of size ${\numelts} + 1$ that can be obtained by a specified procedure from the object of size {\numelts} represented by node {\parentnode}. Let $\children(\parentnode)$ represent the children of node {\parentnode}.

We will begin by applying generating tree analysis to linear extensions of {\combB{\numteeth}{\toothlen}}. We will then apply it to $312$-avoiding permutations,
and finally, we will combine the two and compare the resulting generating tree to the generating tree for $(\toothlen+1)$-ary trees.

\subsubsection{Generating linear extensions of type-{\typeB} combs.}
\label{sec:extgentree}
Say we have a linear extension of {\unevencombB{\toothlen}{\numelts}} and we want to insert the element $\numelts +1$ into the linear extension in such a way that the resulting permutation will be a linear extension of {\unevencombB{\toothlen}{\numelts + 1}}. 
The constraints for doing this are as follows:
\begin{itemize}
\item If {\numelts} is divisible by {\toothlen}, then $\numelts + 1$ must be inserted after the greatest element of the form ${1 \bmod{\toothlen}}$ already in the permutation, since the elements of the form ${1 \bmod{\toothlen}}$ form the spine of {\unevencombB{\toothlen}{\numelts}} and the spinal minimal relations require that they appear in order.
\item If {\numelts} is not divisible by {\toothlen}, then $\numelts + 1$ must be inserted after {\numelts}, since unless $\numelts+1$ is located on the spine of the comb, the tooth minimal relations require that $\numelts + 1$ appears after {\numelts}.
\end{itemize}

\subsubsection{Generating $312$-avoiding permutations.} 
\label{sec:avgentree}
Now, say we have a $312$-avoiding permutation $\perm \in S_{\numelts}$, and we want to insert the element $\numelts + 1$ in such a way that the resulting permutation $\perm' \in S_{\numelts + 1}$ will still be $312$-avoiding. 
$\numelts+1$ may be inserted anywhere as long as it is not followed by an increasing subsequence (elements $\someelt_1$ and $\someelt_2$ such that $\someelt_1$ appears before $\someelt_2$ and $\someelt_1 < \someelt_2$). 

\subsubsection{Generating $312$-avoiding linear extensions of type-{\typeB} combs.}
Now, let's combine the two scenarios above; say we have a $312$-avoiding linear extension of {\unevencombB{\toothlen}{\numelts}} and we want to insert the element $\numelts + 1$ into the linear extension in such a way that the resulting permutation will be a $312$-avoiding linear extension of {\unevencombB{\toothlen}{\numelts+1}}.
The two sets of restrictions for next-element-insertion described in Sections \ref{sec:extgentree} and \ref{sec:avgentree} both prohibit insertion into some prefix of the existing permutation. 
By taking the union of these prefixes, we get the prefix into which $\numelts + 1$ cannot be inserted in order to obtain a $312$-avoiding linear extension of {\unevencombB{\toothlen}{\numelts+1}}.
The following describes the rules by which $\numelts + 1$ can be inserted:
\begin{itemize}
\item If {\numelts} is divisible by {\toothlen}, then we can insert $\numelts + 1$ anywhere after both of the following:
\begin{itemize}
\item the first element of the last increasing subsequence (elements $\someelt_1$ and $\someelt_2$ such that $\someelt_1$ appears before $\someelt_2$ and $\someelt_1 < \someelt_2$), and
\item the last element of the form $\someint \toothlen + 1$ (for integer {\someint}).
\end{itemize}
\item If {\numelts} is not divisible by {\toothlen}, then we can insert $\numelts + 1$ anywhere after both of the following:
\begin{itemize}
\item the first element of the last increasing subsequence, and 
\item the element {\numelts}.
\end{itemize}
\end{itemize}

\subsubsection{Generating tree for $312$-avoiding linear extensions of {\unevencombB{\toothlen}{\numelts}}.}
\label{sec:312avcombgentree}

Let's fix a positive integer {\toothlen} and consider a generating tree for $312$-avoiding linear extensions of {\unevencombB{\toothlen}{\numelts}}. 
Let each node of the generating tree represent a $312$-avoiding linear extension of {\unevencombB{\toothlen}{\numelts}}, and define the children of a node representing $\perm \in S_\numelts$ to be all $312$-avoiding linear extensions of {\unevencombB{\toothlen}{\numelts + 1}} that can be obtained by inserting the element $\numelts+1$ into {\perm}. The root node will represent the empty permutation and will have a single child, a node representing the permutation $[1] $. For $\numelts \leq \toothlen$, the $\numelts^{\text{th}}$ level will always contain exactly one node; however, at the $\toothlen^{\text{th}}$ level, the tree starts to branch out. Figure \ref{fig:312avtree} shows the first few levels of the generating tree for $312$-avoiding linear extensions of {\unevencombB{2}{\numelts}}.

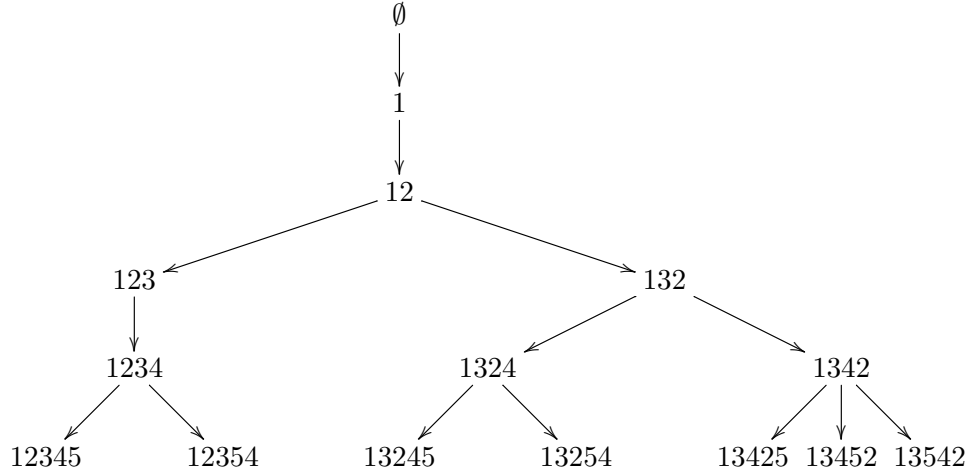
\begin{figure}[H]
\begin{displaymath}
    \xymatrix@!=2ex{
       & & & & \emptyset \ar[d] & & & & & & &\\
       & & & & 1 \ar[d] & & & & & & &\\
       & & & & 12 \ar[drrr] \ar[dlll] & & & & & & & \\
       & 123 \ar[d] & & & & & & 132 \ar[drr] \ar[dll] & & & &\\
       & 1234 \ar[dr] \ar[dl] & & & & 1324 \ar[dr] \ar[dl] & & & & 1342 \ar[dr] \ar[d] \ar[dl]& &\\
       12345 & & 12354 & & 13245 & & 13254 & & 13425 & 13452 & 13542 &
       }
\end{displaymath}
\caption{First few levels of the generating tree for $312$-avoiding linear extensions {\unevencombB{2}{\numelts}}}
\label{fig:312avtree}
\end{figure}

Consider a $312$-avoiding linear extension {\perm} of {\unevencombB{\toothlen}{\numelts}}. The element $\numelts+1$ can be inserted anywhere before, after or among the last $|\children(\perm)|-1$ elements of {\perm}. 
Say that we derive $\perm'$ from {\perm} by inserting $\numelts+1$ in such a way that it precedes $\someint$ other elements, where $\someint \in \{0, \ldots, |\children(\perm)| -1 \}$. Note the following:
\[
|\children(\perm')| =  \left\{ \begin{array}{rl}
  \someint + 2 &\mbox{ if {\toothlen} divides $\numelts+1$, making $\numelts+2$ a spinal element;} \\
  \someint + 1&\mbox{ if {\toothlen} does not divide $\numelts+1$.}
       \end{array} \right.
\]

This holds because if $n+2$ is a spinal element, it can be inserted before or after $n+1$, whereas if $n+2$ is not a spinal element, the minimal relations of {\unevencombB{\toothlen}{\numelts + 2}} dictate that it appears after $n+1$.

By representing a generating tree node at the {\numelts}th level by its number of children {\numchildren}, we can write the propagation rules of our generating tree as follows:

\[
 (\numchildren)
 \rightsquigarrow
 \left\{ \begin{array}{rl}
 (2) \ldots (\numchildren+1) &\mbox{ if {\toothlen} divides $\numelts+1$, making $\numelts+2$ a spinal element;} \\
 (1) \ldots (\numchildren) &\mbox{ if {\toothlen} does not divide $\numelts+1$,}
       \end{array} \right.
\]
where $(\numchildren)$ represents a tree node with {\numchildren} children.

\subsubsection{Generating tree for $(\toothlen+1)$-ary trees.}

Now, let's show that the number of nodes at the ($\toothlen \numteeth$)th level of the generating tree for $312$-avoiding linear extensions of {\unevencombB{\toothlen}{\numelts}} is equal to the number of $({\toothlen}+1)$-ary trees on {\numteeth} nodes. It is known (by extension of the work of Pfaff and Fuss in 1791 \cite{Fuss}) that the number of $(\toothlen+1)$-ary trees on {\numteeth} nodes is equal to the number of lattice paths from $(0,0)$ to $(\toothlen \numteeth+1, \numteeth)$ composed of steps $(0,1)$ and $(1,0)$ that do not cross (but may touch) the line $y= x/{\toothlen}$, and start with the step $(1, 0)$.
We will show the equivalence between the generating tree for these lattice paths and the generating tree described in Section \ref{sec:312avcombgentree} above. 

Let's begin by sketching out a generating tree for the lattice paths.
Let a child of a node representing a path leading to $(x,y)$ be any node representing that same path followed by zero or more (legal) steps up and one step across (in that order). It is clear that any legal path can be described by a node in such a tree.

If node {\parentnode} represents a path leading to $(x,y)$ such that $0 \leq x \leq \toothlen y$, the number of points that can be reached by taking zero or more steps up without crossing the $y = x/{\toothlen}$ line is $\lfloor x/{\toothlen} \rfloor - y$, so {\parentnode} will have $\lfloor x/{\toothlen} \rfloor - y$ children.

\begin{figure}[H]
\begin{displaymath}
    \xymatrix@!=2ex{
       & & & & (0,0) \ar[d] & & & & & & &\\
       & & & & (1,0) \ar[drrr] \ar[dlll] & & & & & & & \\
       & (2,1) \ar[d] & & & & & & (2,0) \ar[drr] \ar[dll] & & & &\\
       & (3,1) \ar[dr] \ar[dl] & & & & (3,1) \ar[dr] \ar[dl] & & & & (3,0) \ar[dr] \ar[d] \ar[dl]& &\\
       (4,2) & & (4,1) & & (4,2) & & (4,1) & & (4,2) & (4,1) & (4,0) &
       }
\end{displaymath}
\caption{First few levels of the generating tree for lattice paths not crossing the $y = x/2$ line}
\end{figure}
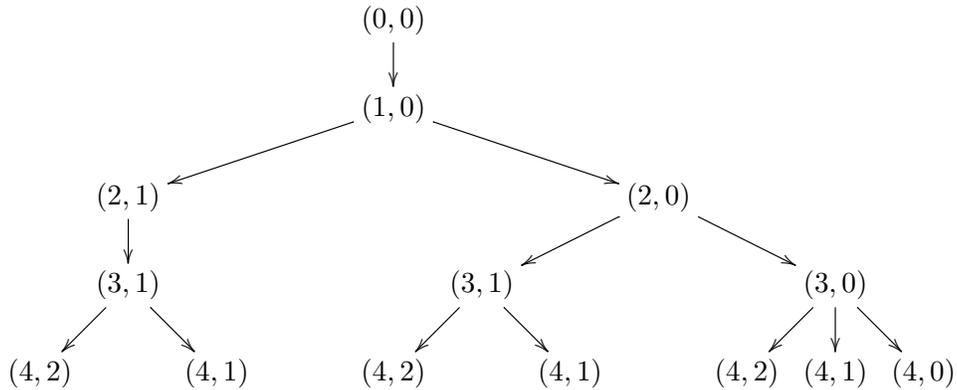

\begin{figure}[H]
\begin{center}
\begin{tikzpicture}[x=1cm,y=1cm]
  \def\xmin{0}
  \def\xmax{5}
  \def\ymin{0}
  \def\ymax{5}

  \draw[style=help lines, ystep=1, xstep=1] (\xmin,\ymin) grid
  (\xmax,\ymax);

  \draw[->] (\xmin,\ymin) -- (\xmax,\ymin) node[right] {$x$};
  \draw[->] (\xmin,\ymin) -- (\xmin,\ymax) node[above] {$y$};

  \foreach \x in {0,1,...,5}
    \node at (\x, \ymin) [below] {\x};
  \foreach \y in {0,1,...,5}
    \node at (\xmin,\y) [left] {\y};
   
   
   \draw[very thick] (0,0) -> (1,0);
   \draw[very thick] (1,0) -> (2,0);
   \draw[very thick] (2,0) -> (2,1);
   \draw[very thick] (2,1) -> (3,1);
   \draw[very thick] (3,1) -> (4,1);
   
   \draw[dashed] (\xmin,\ymin) -- (\xmax,\ymax/2);
   
    \begin{customlegend}[legend entries={the y = x/2 line, the lattice path}, legend style={at={(4,4)},anchor=center}]
    \addlegendimage{dashed, sharp plot}
    \addlegendimage{very thick, sharp plot}
    \end{customlegend}

\end{tikzpicture}
\end{center}
\caption{Lattice path corresponding to $(0,0) \rightarrow (1,0) \rightarrow (2,1) \rightarrow (3,1) \rightarrow (4,1)$}
\end{figure}
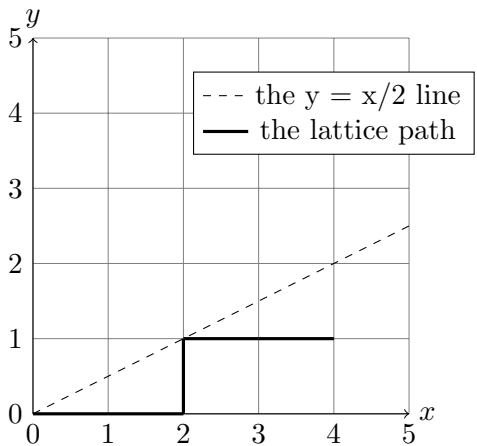

Let $|\children(x,y)|$ denote the number of children of a node representing a path leading to $(x,y)$. Note that though there may be multiple tree nodes representing paths leading to $(x, y)$, they will all have the same number of children.

\[
|\children(x+1,y)| = \left\{ \begin{array}{rl}
|\children(x, y)| + 1&\mbox{ if {\toothlen} divides $x+1$;} \\
|\children(x, y)| &\mbox{ if {\toothlen} does not divide $x+1$.}
       \end{array} \right.
\]

This is true because when $x+1$ is a multiple of {\toothlen}, $y= (x+1)/{\toothlen}$ is an integer, allowing one additional step up that was not possible at $x$.

By representing a generating tree node at the $x$th level by its number of children {\numchildren}, we can write the propagation rules of our generating tree as follows:

\[
 (\numchildren)
 \rightsquigarrow
 \left\{ \begin{array}{rl}
 (2) \ldots (\numchildren+1) &\mbox{ if {\toothlen} divides $x+1$;} \\
 (1) \ldots (\numchildren) &\mbox{ if {\toothlen} does not divide $x+1$,}
       \end{array} \right.
\]
where $(\numchildren)$ represents a tree node with {\numchildren} children.

Recall that the generating tree described in Section \ref{sec:312avcombgentree} above propagates following the exact same pattern. The root of each tree has one child, so  the two trees must be identical. This concludes the proof.

\end{proof}

\section{Avoidance of length $3$ patterns in type-{\typeA} comb linear extensions}
\label{sec:avalphaone}

This section goes through the enumerations of type-{\typeA} comb linear extensions avoiding each of the length-3 patterns $\pattern \in S_3$. 
For some patterns (such as $123$ and $132$), enumeration is trivial; for others, it is more interesting; and for yet others (such as $231$, $312$ for $\toothlen > 2$ and $321$), enumeration remains an open problem.
The tables in this section depict the observed values of {\numavexts{\pattern}{\combA{\numteeth}{\toothlen}}} for each pattern $\pattern \in S_3$. 
Tables are omitted for patterns {\pattern} for which enumeration was successful.

\subsection{Type-{\typeA} $123$-avoidance}

\begin{thm}
\label{thm:numavexts123combA}
${\numavexts{123}{\combA{\numteeth}{\toothlen}}} = 0$ for $\toothlen > 1$ and $\numteeth > 1$. 
\end{thm}

This is apparent, because the minimal relations of {\combA{\numteeth}{\toothlen}} always force a $[1, 2, 3]$ pattern. 

\subsection{Type-{\typeA} $132$-avoidance}

\begin{thm}
\label{thm:numavexts132combA}
${\numavexts{132}{\combA{\numteeth}{\toothlen}}} = 1$
\end{thm}

This is also apparent; the only such linear extension is the permutation $[1, 2, 3, \ldots, \numelts]$.

\subsection{Type-{\typeA} $213$-avoidance}
\label{sec:numavexts213combA}

\proventable{
\begin{figure}[H]
\begin{center}
\begin{tabular}{c | c | c | c}
  \hline
  {\numteeth} &  
{\numavexts{213}{\combA{\numteeth}{2}}}  &   
{\numavexts{213}{\combA{\numteeth}{3}}} &
{\numavexts{213}{\combA{\numteeth}{4}}} \\ \hline
2 & \observedval{2} & \observedval{2} & \observedval{2} \\ 
3 & \observedval{5} & \observedval{5} & \observedval{5}  \\ 
4 & \observedval{14} & \observedval{14} & \extrapval{14}  \\ 
5 & \observedval{42} & \extrapval{42} & \extrapval{42}  \\ 
6 & \observedval{132} & \extrapval{132} & \extrapval{132} \\ 
  \end{tabular}
\end{center}
\caption{{\numavexts{213}{\combA{\numteeth}{\toothlen}}} for various values of {\toothlen} and {\numteeth}}
\end{figure}

This sequence - $2,5,14,42,132, \ldots$ - is easily recognized as the Catalan sequence.
}

\begin{thm}
\label{thm:numavexts213combA}
${\numavexts{213}{\combA{\numteeth}{\toothlen}}} = C_{\numteeth}$ for $\toothlen > 1$, where $C_{\numteeth}$ is the {\numteeth}th Catalan number.
\end{thm}

\begin{proof}

First note that the first $\numteeth (\toothlen - 1)$ elements of any $213$-avoiding permutation extending {\combA{\numteeth}{\toothlen}} must appear consecutively and in order. 
We can show this by contradiction: assume that some element $\someint \in \{1, \ldots,  \numteeth (\toothlen - 1)\}$ is preceded by a greater element $\someint'$. 
(Note that $\someint'$ and {\someint} cannot be in the same tooth, because if they were, the minimal relations of {\combA{\numteeth}{\toothlen}} would force them to appear in order.)
We then have two cases: 
\begin{enumerate}
\item \label{item:case1}
$\someint' \leq \numteeth (\toothlen - 1)$, or 
\item \label{item:case2}
$\someint' > \numteeth (\toothlen - 1)$.
\end{enumerate}

If case \ref{item:case1} holds, then $[\someint', \someint, \someint'']$ forms a $[2, 1, 3]$ pattern, where $\someint''$ is the last element in the tooth to which {\someint} belongs.
({\someint} cannot be the last element of its tooth, since $\someint < \numteeth (\toothlen - 1)$, and only the last $\numteeth$ elements form tooth ends.)

If case \ref{item:case2} holds, then we have the constraint that $\someint' - (\toothlen-1)\numteeth \leq_{\plaincomb} \cdots \leq_{\plaincomb} \someint' - 2\numteeth \leq_{\plaincomb} \someint' - \numteeth \leq_{\plaincomb} \someint'$, so all of $\{\someint' - (\toothlen-1)\numteeth, \ldots, \someint' - 2\numteeth, \someint' - \numteeth, \someint'\}$ must appear before {\someint}.
Let $\someint''$ be the last element in the tooth to which $\someint$ belongs. At least one $\someotherint \in \{\someint' - (\toothlen-1)\numteeth, \ldots, \someint' - 2\numteeth, \someint' - \numteeth, \someint'\}$ is such that $\someint \leq \someotherint \leq \someint''$; thus, $[\someotherint, \someint, \someint'']$ forms a $[2, 1, 3]$ pattern.

We have more freedom with the ordering of the last {\numteeth} elements of the extending permutation. They can appear in any order, as long as they avoid $213$. 
We know that there are $C_{\numteeth}$  {\numteeth}-element $213$-avoiding permutations \cite{Knuth}, so we can conclude that ${\numavexts{213}{\combA{\numteeth}{\toothlen}}} = C_{\numteeth}$.
\end{proof}

\subsection{Type-{\typeA} $231$-avoidance}
\label{sec:numavexts231combA}

\begin{figure}[H]
\begin{center}
\begin{tabular}{c | c | c | c}
  \hline
  {\numteeth} &  
\numavexts{231}{\combA{\numteeth}{2}}  &   
\numavexts{231}{\combA{\numteeth}{3}} &
\numavexts{231}{\combA{\numteeth}{4}} \\ \hline
2 & \observedval{3} & \observedval{8} & \observedval{21} \\ 
3 & \observedval{11} & \observedval{91} & \suggestedval{731} \\ 
4 & \observedval{44} & \observedval{1210} &  \\ 
5 & \observedval{185} & \suggestedval{17606} &  \\ 
6 & \observedval{804} & & \\ 
7 & \suggestedval{3579} & & \\
  \end{tabular}
\end{center}
\caption{{\numavexts{231}{\combA{\numteeth}{\toothlen}}} for various values of {\toothlen} and {\numteeth}}
\end{figure}

The values in bold were kindly supplied by an anonymous reviewer.

Enumerating $231$-avoiding linear extensions of {\combA{\numteeth}{\toothlen}} remains an open problem.
According to the Online Encyclopedia of Integer Sequences \cite{oeis}, the observed values of {\numavexts{231}{\combA{\numteeth}{2}}} match the sequence given by the generating function $\frac{1}{1-x \cdot C(x) \cdot C(x \cdot C(x))} = \frac{2}{1 + \sqrt{2 \sqrt{1 - 4x} - 1}}$, where $C(x)$ is the generating function for the Catalan numbers.

{\numavexts{231}{\combA{\numteeth}{3}}} does not appear in the Online Encyclopedia of Integer Sequences.

\subsection{Type-{\typeA} $312$-avoidance}
\label{sec:numavexts312combA}

\begin{figure}[H]
\begin{center}
\begin{tabular}{c | c | c | c}
  \hline
  {\numteeth} &  
{\numavexts{312}{\combA{\numteeth}{2}}}  &   
{\numavexts{312}{\combA{\numteeth}{3}}} &
{\numavexts{312}{\combA{\numteeth}{4}}} \\ \hline
2 & \observedval{3} & \observedval{8} & \observedval{21} \\ 
3 & \observedval{9} & \observedval{73} & \suggestedval{585} \\ 
4 & \observedval{28} & \observedval{738} &  \\ 
5 & \observedval{90} & \suggestedval{8022} &  \\ 
6 & \observedval{297} & & \\ 

  \end{tabular}
\end{center}
\caption{{\numavexts{312}{\comb{\numteeth}{\toothlen}}} for various values of {\toothlen} and {\numteeth}}
\end{figure}

The values in bold were kindly supplied by an anonymous reviewer.

\begin{thm}
\label{thm:numavexts312combA}
${\numavexts{312}{\combA{\numteeth}{2}}} = C_{\numteeth + 1} - C_{\numteeth}$, where $C_{\numteeth}$ is the {\numteeth}th Catalan number.
\end{thm}

\begin{proof}
We know that any linear extension of a type-{\typeA} comb has to contain the elements $\{1, 2, \ldots, \numteeth\}$ in order. Now, notice that $\{1, 2, \ldots, \numteeth - 1\}$ must also appear consecutively, since if $\someint > \numteeth$ appears before $\numteeth - 1$, then $[\someint, \numteeth - 1, \numteeth]$ will form a $[3,1,2]$ pattern.
We see that the rest of the linear extension consists of $\{\numteeth, \numteeth + 1, \ldots, 2\numteeth \}$ in any $312$-avoiding order which has {\numteeth} preceding $2 \numteeth$.
Note that any element preceding {\numteeth} must be smaller than any element following {\numteeth}, since otherwise a $[3,1,2]$ pattern will occur, with {\numteeth} as the $1$. 
Suppose that $\{\numteeth + 1, \numteeth + 2, \ldots, \numteeth + i\}$ precede {\numteeth}, and $\{\numteeth + i + 1, \numteeth + i + 2, \ldots, 2\numteeth\}$ follow {\numteeth}.
$\{\numteeth + 1, \numteeth + 2, \ldots, \numteeth + i\}$ can be arranged in any $312$-avoiding permutation, and so can $\{\numteeth + i + 1, \numteeth + i + 2, \ldots, 2\numteeth\}$.\\
Recall that there are $C_{i}$ $312$-avoiding permutations of $\{1, \ldots, i\}$. So, there are $C_{i}C_{\numteeth-i}$ ways for $\{\numteeth + 1, \numteeth + 2, \ldots, \numteeth + i\}$ to precede {\numteeth} in a $312$-avoiding way and $\{\numteeth + i + 1, \numteeth + i + 2, \ldots, 2\numteeth\}$ to follow {\numteeth} in a $312$-avoiding way.
So, 
\[A_{312}(\combA{\numteeth}{2}) = \sum_{i = 0}^{\numteeth - 1} C_{i}C_{\numteeth - i}\]
\[ = \sum_{i  = 0}^{\numteeth}C_{i}C_{\numteeth - i} - C_{\numteeth}C_{0}\]
\[ = C_{\numteeth + 1} - C_{\numteeth}.\]
\end{proof}

Enumerating $312$-avoiding linear extensions of {\combA{\numteeth}{\toothlen}} for $\toothlen > 2$ remains an open problem.

\subsection{Type-{\typeA} $321$-avoidance}
\label{sec:numavexts321combA}

\begin{figure}[H]
\begin{center}
\begin{tabular}{c | c | c | c}
  \hline
  {\numteeth} &  
{\numavexts{321}{\combA{\numteeth}{2}}}  &   
{\numavexts{321}{\combA{\numteeth}{3}}} &
{\numavexts{321}{\combA{\numteeth}{4}}} \\ \hline
2 & \observedval{3} & \observedval{10} & \observedval{35} \\ 
3 & \observedval{13} & \observedval{161} &  \\ 
4 & \observedval{67} & \observedval{3196} &  \\ 
5 & \observedval{378} & &  \\ 
6 & \observedval{2244} & & \\ 
  \end{tabular}
\end{center}
\caption{{\numavexts{321}{\combA{\numteeth}{\toothlen}}} for various values of {\toothlen} and {\numteeth}}
\end{figure}

Enumerating $321$-avoiding linear extensions of {\combA{\numteeth}{\toothlen}} remains an open problem. 

\section{Avoidance of other length $3$ patterns in type-{\typeB} comb linear extensions}
\label{sec:avbetaone}

This section goes through the enumerations of type-{\typeB} comb linear extensions avoiding each of the length-3 patterns $\pattern \in S_3$. 
As with type-{\typeA} combs, enumerating {\pattern}-avoiding type-{\typeB} comb linear extensions is trivial for some patterns {\pattern} (such as $123$ and $132$) and more interesting for other patterns.
The tables in this section depict in bold the observed values of {\numavexts{\pattern}{\combA{\numteeth}{\toothlen}}} for each pattern $\pattern \in S_3$. For the patterns {\pattern} for which enumeration was successful, the remaining values are filled in.

\subsection{Type-{\typeB} $123$-avoidance}

\begin{thm}
\label{thm:numavexts123combB}
${\numavexts{123}{\combB{\numteeth}{\toothlen}}} = 0$ for $\toothlen > 1$ and $\numteeth > 1$. 
\end{thm}

This is apparent, because the minimal relations of {\combB{\numteeth}{\toothlen}} always force a $[1, 2, 3]$ pattern. 

\subsection{Type-{\typeB} $132$-avoidance}

\begin{thm}
\label{thm:numavexts132combB}
${\numavexts{132}{\combB{\numteeth}{\toothlen}}} = 1$.
\end{thm}

This is also apparent; the only such linear extension is the permutation $[1, 2, 3, \ldots, \numelts]$.

\subsection{Type-{\typeB} $213$-avoidance}

\proventable{
\begin{figure}[H]
\begin{center}
\begin{tabular}{c | c | c | c}
  \hline
  {\numteeth} &  
{\numavexts{213}{\combB{\numteeth}{2}}}  &   
{\numavexts{213}{\combB{\numteeth}{3}}} &
{\numavexts{213}{\combB{\numteeth}{4}}} \\ \hline
2 & \observedval{2} & \observedval{3} & \observedval{4} \\ 
3 & \observedval{4} & \observedval{9} & \extrapval{16} \\ 
4 & \observedval{8} & \observedval{27} & \extrapval{64} \\ 
5 & \observedval{16} & \extrapval{81} & \extrapval{256} \\ 
6 & \observedval{32} & \extrapval{243} & \extrapval{1024} \\ 
  \end{tabular}
\end{center}
\caption{{\numavexts{213}{\combB{\numteeth}{\toothlen}}} for various values of {\toothlen} and {\numteeth}}
\end{figure}
}

\begin{thm}
\label{thm:numavexts213combB}
$\numavexts{213}{\combB{\numteeth}{\toothlen}} = \toothlen^{\numteeth - 1}$.
\end{thm}

\begin{proof}
We will show this using induction. 

Our base case, {\combB{1}{\toothlen}}, has only one linear extension; $[1, 2, \ldots, \toothlen]$. This linear extension avoids $213$, so we can conclude that $\numavexts{213}{\combB{1}{\toothlen}} = 1$.

Next, we will show that $\numavexts{213}{\combB{\numteeth + 1}{\toothlen}} = \toothlen \times \numavexts{213}{\combB{\numteeth}{\toothlen}}$. 
We will do this by demonstrating that every $213$-avoiding linear extension of {\combB{\numteeth}{\toothlen}} leads to {\toothlen} $213$-avoiding linear extensions of {\combB{\numteeth + 1}{\toothlen}}.
Let $E = [\someelt_1, \ldots, \someelt_{\toothlen \numteeth}]$ be a $213$-avoiding linear extension of {\combB{\numteeth}{\toothlen}}.
Let $E' = [\someelt_1', \ldots, \someelt_{\toothlen \numteeth}']$ be $[\someelt_1 + \toothlen, \ldots, \someelt_{\toothlen \numteeth} + \toothlen]$. Thus, $E'$ is a $213$-avoiding permutation of $[1 + \toothlen, 2 + \toothlen, \ldots, (\numteeth + 1) \toothlen]$. 
$E'$ can be turned into a $213$-avoiding linear extension of {\combB{\numteeth + 1}{\toothlen}} by appending $1$ before $\someelt_1'$ and inserting $2, \ldots, \toothlen$ after $1$ in any way such that their relative order is maintained and the resulting permutation is still $213$-avoiding. 
It turns out that $\otherelt \in \{2, \ldots, \toothlen\}$ can only be inserted in two places: directly between $1$ and $\someelt_1'$ or at the very end. 
Otherwise, $[\someelt_1', \otherelt, \someelt_{\toothlen  \numteeth}']$ would form a $[2, 1, 3]$ pattern (recall that $\someelt_1' = 1 + \toothlen$ and $\someelt_{\toothlen \numteeth}' > 1 + \toothlen$).
Since $\{2, \ldots, \toothlen\}$ belong to a tooth and thus must appear in order, there are exactly {\toothlen} ways of placing them in the two possible locations.
\end{proof}

\subsection{Type-{\typeB} $231$-avoidance}

\proventable{
\begin{figure}[H]
\begin{center}
\begin{tabular}{c | c | c | c}
  \hline
  {\numteeth} &  
{\numavexts{231}{\combB{\numteeth}{2}}}  &   
{\numavexts{231}{\combB{\numteeth}{3}}} &
{\numavexts{231}{\combB{\numteeth}{4}}} \\ \hline
2 & \observedval{2} & \observedval{3} & \observedval{4} \\ 
3 & \observedval{4} & \observedval{9} & \extrapval{16} \\ 
4 & \observedval{8} & \observedval{27} & \extrapval{64} \\ 
5 & \observedval{16} & \extrapval{81} & \extrapval{256} \\ 
6 & \observedval{32} & \extrapval{243} & \extrapval{1024} \\ 
  \end{tabular}
\end{center}
\caption{{\numavexts{231}{\combB{\numteeth}{\toothlen}}} for various values of {\toothlen} and {\numteeth}}
\end{figure}
}

\begin{thm}
\label{thm:numavexts231combB}
$\numavexts{231}{\combB{\numteeth}{\toothlen}} = \toothlen^{\numteeth - 1}$.
\end{thm}

Theorem \ref{thm:numavexts231combB} can be proven with a slight modification on the proof of Theorem \ref{thm:numavexts213combB}.
It can be shown that $\numavexts{231}{\combB{\numteeth + 1}{\toothlen}} = \toothlen \times \numavexts{231}{\combB{\numteeth}{\toothlen}}$ by showing that $[\someelt_{\toothlen \numteeth}, \numteeth \toothlen + 1, \ldots, (\numteeth + 1)\toothlen]$ must be the last $\toothlen + 1$ elements of any $231$-avoiding extension of \combB{\numteeth + 1}{\toothlen}, and that all $\toothlen + 1$ of these elements must appear in increasing order with the exception of one of $\toothlen$ pairs of adjacent elements, whose order may be reversed.

\subsection{Type-{\typeB} $312$-avoidance}

This result was discussed in Section \ref{sec:numavexts312combB}.

\subsection{Type-{\typeB} $321$-avoidance}

\begin{figure}[H]
\begin{center}
\begin{tabular}{c | c | c | c}
  \hline
  {\numteeth} &  
{\numavexts{321}{\combB{\numteeth}{2}}} &   
{\numavexts{321}{\combB{\numteeth}{3}}} &
{\numavexts{321}{\combB{\numteeth}{4}}} \\ \hline
2 & \observedval{3} & \observedval{10} & \observedval{35} \\ 
3 & \observedval{12} & \observedval{127} & \observedval{1222} \\ 
4 & \observedval{55} & \observedval{1866} &  \\ 
5 & \observedval{273} & &  \\ 
6 & \observedval{1428} & & \\ 
  \end{tabular}
\end{center}
\caption{{\numavexts{321}{\combB{\numteeth}{\toothlen}}} for various values of {\toothlen} and {\numteeth}}
\end{figure}

A slightly modified version of the argument made in Section {\ref{sec:numavexts312combB}} can prove that $\numavexts{321}{\combB{\numteeth}{2}} = \numavexts{312}{\combB{\numteeth}{2}}$. 
We simply consider insertions after the first element of the last falling subsequence as opposed to the last rising subsequence, and the rest of the proof follows.

Enumerating $321$-avoiding linear extensions of {\unevencombB{\numteeth}{\toothlen}} for $\toothlen > 2$ remains an open problem. 

\section{Avoidance of multiple length-3 patterns in type-{\typeA} comb linear extensions}
\label{sec:avalphatwo}

In this section, we will explore type-{\typeA} comb linear extensions which avoid two length-$3$ patterns at once.
Extending the notation introduced in previous sections, let $\numavexts{\pattern, \otherpattern}{\poset}$ denote the number of linear extensions of the poset {\poset} that avoid both {\pattern} and {\otherpattern}.

For some choices of patterns, enumerating such linear extensions is trivial. For instance, since for $\numteeth > 1$ there are no comb linear extensions of type either {\typeA} or {\typeB} avoiding $123$, it follows that there are also no comb linear extensions avoiding both $123$ and another pattern.
Recall also that there is always exactly one comb linear extension of type both {\typeA} and {\typeB} avoiding $132$, and that is $[1, 2, \ldots, \numteeth \toothlen]$. Since this linear extension also avoids all of $\{213, 231, 312, 321\}$, it follows that there is exactly one comb linear extension of type both {\typeA} and {\typeB} which avoids $132$ and one of $\{213, 231, 312, 321\}$.

\subsection{Type-{\typeA} $213$- and $231$-avoidance, and type-{\typeA} $213$- and $312$-avoidance}

\proventable{
\begin{figure}[H]
\begin{center}
\begin{tabular}{c | c | c | c}
  \hline
  {\numteeth} &  
{\numavexts{213, 231}{\combA{\numteeth}{2}}} & 
{\numavexts{213, 231}{\combA{\numteeth}{3}}} &
{\numavexts{213, 231}{\combA{\numteeth}{4}}} \\ 
 &
= {\numavexts{213, 312}{\combA{\numteeth}{2}}} &   
= {\numavexts{213, 312}{\combA{\numteeth}{3}}} &   
= {\numavexts{213, 312}{\combA{\numteeth}{4}}} \\ \hline
2 & \observedval{2} & \observedval{2} & \extrapval{2} \\ 
3 & \observedval{4} & \observedval{4} & \extrapval{4} \\ 
4 & \observedval{8} & \observedval{8} & \extrapval{8} \\ 
5 & \observedval{16} & \extrapval{16} & \extrapval{16} \\ 
6 & \observedval{32} & \extrapval{32} & \extrapval{32} \\ 
  \end{tabular}
\end{center}
\caption{{\numavexts{213, 231}{\combA{\numteeth}{\toothlen}}} = {\numavexts{213, 312}{\combA{\numteeth}{\toothlen}}} for various values of {\toothlen} and {\numteeth}}
\end{figure}
}

\begin{thm}
\label{thm:numavexts213and231combA}
$\numavexts{213, 231}{\combA{\numteeth}{\toothlen}} = \numavexts{213, 312}{\combA{\numteeth}{\toothlen}} =  2^{\numteeth - 1}.$
\end{thm}

\begin{proof}
Recall from the proof of Theorem \ref{thm:numavexts213combA} that in order for a linear extension of a type-$\typeA$ comb to avoid $213$, the first $(\toothlen - 1)\numteeth$ elements must appear consecutively in order. It is known that  that the number of ($213, 231$)-avoiding elements of $S_{\numteeth}$ (similarly, the number of ($213, 312$)-avoiding elements of $S_{\numteeth}$) is $2^{\numteeth-1}$  \cite{simionschmidt}, so we can conclude that there are $2^{\numteeth-1}$ ways to order the remaining elements.
\end{proof}

\subsection{Type-{\typeA} $213$- and $321$-avoidance}

\proventable{
\begin{figure}[H]
\begin{center}
\begin{tabular}{c | c | c | c}
  \hline
  {\numteeth} &  
{\numavexts{213, 321}{\combA{\numteeth}{2}}}  &   
{\numavexts{213, 321}{\combA{\numteeth}{3}}} &
{\numavexts{213, 321}{\combA{\numteeth}{4}}} \\ \hline
2 & \observedval{2} & \observedval{2} & \extrapval{2} \\ 
3 & \observedval{4} & \observedval{4} & \extrapval{4} \\ 
4 & \observedval{7} & \observedval{7} & \extrapval{7} \\ 
5 & \observedval{11} & \extrapval{11} & \extrapval{11} \\ 
6 & \observedval{16} & \extrapval{16} & \extrapval{16} \\ 
  \end{tabular}
\end{center}
\caption{{\numavexts{213, 321}{\combA{\numteeth}{\toothlen}}} for various values of {\toothlen} and {\numteeth}}
\end{figure}
}

\begin{thm}
\label{thm:numavexts213and321combA}
$\numavexts{213, 321}{\combA{\numteeth}{\toothlen}} = {\numteeth \choose 2} + 1.$
\end{thm}

\begin{proof}
As above, recall from the proof of Theorem \ref{thm:numavexts213combA} that in order for a linear extension of a type-$\typeA$ comb to avoid $213$, the first $(\toothlen - 1)\numteeth$ elements must appear consecutively in order. It is known that the number of ($213, 321$)-avoiding elements of $S_{\numteeth}$ is ${\numteeth \choose 2} + 1$ \cite{simionschmidt}, so we can conclude that there are ${\numteeth \choose 2} + 1$ ways to order the remaining elements.
\end{proof}

\subsection{Type-{\typeA} $231$- and $312$-avoidance}

\begin{figure}[H]
\begin{center}
\begin{tabular}{c | c | c | c}
  \hline
  {\numteeth} &  
{\numavexts{231, 312}{\combA{\numteeth}{2}}} &   
{\numavexts{231, 312}{\combA{\numteeth}{3}}} &
{\numavexts{231, 312}{\combA{\numteeth}{4}}} \\ \hline
2 & \observedval{3} & \observedval{8} & \extrapval{21} \\ 
3 & \observedval{7} & \observedval{44} & \extrapval{274} \\ 
4 & \observedval{15} & \observedval{208} & \extrapval{2872} \\ 
5 & \observedval{31} & \extrapval{912} & \extrapval{26784} \\ 
6 & \observedval{63} & \extrapval{3840} & \extrapval{233904} \\ 
  \end{tabular}
\end{center}
\caption{{\numavexts{231, 312}{\combA{\numteeth}{\toothlen}}} for various values of {\toothlen} and {\numteeth}}
\end{figure}

\begin{thm}
\label{thm:numavexts231and312combA}

\[
\numavexts{231, 312}{\unevencombA{\numteeth}{\numelts}} = \left\{ \begin{array}{ll}
 1 & \mbox{ if $\numelts \leq \numteeth$,} \\
 2^{\numelts - \numteeth} & \mbox{ if $\numteeth < \numelts < 2\numteeth$,} \\
 2 \cdot \numavexts{231, 312}{\unevencombA{\numteeth}{\numelts - 1}} - & \\
 \mbox{\hspace{10 mm}} \numavexts{231, 312}{\unevencombA{\numteeth}{\numelts - \numteeth - 1}} & \mbox{ if $2\numteeth \leq \numelts$.}
\end{array} \right.  
\]
\end{thm}

\begin{proof}
In any ($231, 312$)-avoiding permutation, any decreasing subsequence must be consecutive (i.e., it must be of the form $\someelt, \someelt-1, \someelt-2, \ldots$).
Note that all of our consecutive decreasing subsequences must appear in increasing order; otherwise, we risk the formation of a non-consecutive decreasing subsequence. 
Since the comb relations require that $\{1, 2, \ldots, \numteeth\}$ appear in increasing order, let's start with the permutation $[1, 2, \ldots, \numteeth]$ and sequentially insert the smallest remaining element of $\{\numteeth + 1, \ldots, \numteeth \toothlen\}$. 
Each new element {\numelts} can only be inserted either directly before the last decreasing subsequence in the existing permutation or in the ultimate position, unless the existing permutation ends with $\numelts - \numteeth$, in which case it can only be inserted in the ultimate position, since $\numelts - \numteeth \leq_{\plaincomb} \numelts$.
This occurs zero times for $\numteeth < \numelts < 2 \numteeth$ (because then we have $\numelts - \numteeth \leq_{\plaincomb} \numteeth$), and exactly 
$\numavexts{231, 312}{\unevencombA{\numteeth}{\numelts - \numteeth - 1}}$ 
times for $\numelts \geq 2 \numteeth$ (since $\numelts - \numteeth$ must have been inserted in the ultimate position and all of $\{\numelts - \numteeth + 1, \ldots, \numelts -1\}$ must have been inserted directly before the last decreasing subsequence in order for $\numelts - \numteeth$ to still be in the ultimate position). 
This shows the recurrence relation stated in Theorem \ref{thm:numavexts231and312combA}.
\end{proof}

In particular, notice that Theorem \ref{thm:numavexts231and312combA} implies that $\numavexts{231, 312}{\combA{\numteeth}{2}} = 2^{\numteeth} - 1$.

\subsection{Type-{\typeA} $231$- and $321$-avoidance}

\begin{figure}[H]
\begin{center}
\begin{tabular}{c | c | c | c}
  \hline
  {\numteeth} &  
{\numavexts{231, 321}{\combA{\numteeth}{2}}}  &   
{\numavexts{231, 321}{\combA{\numteeth}{3}}} &
{\numavexts{231, 321}{\combA{\numteeth}{4}}} \\ \hline
2 & \observedval{3} & \observedval{8} &  \\ 
3 & \observedval{9} & \observedval{57} &  \\ 
4 & \observedval{25} & \observedval{349} &  \\ 
5 & \observedval{65} &  &  \\ 
6 &  &  &  \\ 
  \end{tabular}
\end{center}
\caption{{\numavexts{231, 321}{\combA{\numteeth}{\toothlen}}} for various values of {\toothlen} and {\numteeth}}
\end{figure}

\begin{conjecture}
\label{conjecture:numavexts231and321combAtoothlen2}
$\numavexts{231, 321}{\combA{\numteeth}{2}} = (\numteeth - 1)2^{\numteeth - 1} + 1.$
\end{conjecture}

Conjecture \ref{conjecture:numavexts231and321combAtoothlen2} has not been proven, and moreover, 
enumerating ($231, 321$)-avoiding linear extensions of {\combA{\numteeth}{\toothlen}} for $\toothlen > 2$ remains an open problem.

\subsection{Type-{\typeA} $312$- and $321$-avoidance}

\begin{figure}[H]
\begin{center}
\begin{tabular}{c | c | c | c}
  \hline
  {\numteeth} &  
{\numavexts{312, 321}{\combA{\numteeth}{2}}}  &   
{\numavexts{312, 321}{\combA{\numteeth}{3}}} &
{\numavexts{312, 321}{\combA{\numteeth}{4}}} \\ \hline
2 & \observedval{3} & \observedval{8} & \extrapval{21} \\ 
3 & \observedval{7} & \observedval{44} & \extrapval{274} \\ 
4 & \observedval{15} & \observedval{208} & \extrapval{2872} \\ 
5 & \observedval{31} & \extrapval{912} &\extrapval{26784} \\ 
6 & \observedval{63} & \extrapval{3840} & \extrapval{233904} \\ 
  \end{tabular}
\end{center}
\caption{{\numavexts{312, 321}{\combA{\numteeth}{\toothlen}}} for various values of {\toothlen} and {\numteeth}}
\end{figure}

\begin{thm}
\label{thm:numavexts312and321combA}
\[
\numavexts{312, 321}{\unevencombA{\numteeth}{\numelts}}
 = \left\{ \begin{array}{ll}
 1 & \mbox{ if $\numelts \leq \numteeth$,} \\
 2^{\numelts - \numteeth} & \mbox{ if $\numteeth < \numelts < 2\numteeth$,} \\
 2 \cdot \numavexts{231, 321}{\unevencombA{\numteeth}{\numelts - 1}} -  & \\  
\hspace{10 mm} \numavexts{231, 321}{\unevencombA{\numteeth}{\numelts - \numteeth - 1}} & \mbox{ if $2\numteeth \leq \numelts$.}
\end{array} \right.  
\]
\end{thm}

The proof for Theorem \ref{thm:numavexts312and321combA} is the same as the proof for Theorem \ref{thm:numavexts231and312combA}, with the distinction that each element being inserted can only appear either in the penultimate or the ultimate position, since in a ($312, 321$)-avoiding permutation, no element can be followed by two elements lower than itself.

\section{Avoidance of multiple length-3 patterns in type-{\typeB} comb linear extensions}
\label{sec:avbetatwo}

\subsection{Type-{\typeB} $213$- and $231$-avoidance}

\begin{thm}
\label{thm:numavexts213and231combB}
$\numavexts{213, 231}{\combB{\numteeth}{\toothlen}} = 1$.
\end{thm}

\begin{proof}
In order for a linear extension of a type-{\typeB} comb to avoid $213$ and $231$, it must hold that the first element of any rising subsequence is not followed by a smaller element. Since the minimal relations of {\combB{\numteeth}{\toothlen}} guarantee that every element is part of a rising subsequence of the form $[\someelt, \someelt + 1]$, this forces all of the elements to appear strictly in increasing order.
\end{proof}

\subsection{Type-{\typeB} $213$- and $312$-avoidance}

\proventable{
\begin{figure}[H]
\begin{center}
\begin{tabular}{c | c | c | c}
  \hline
  {\numteeth} &  
{\numavexts{213, 312}{\combB{\numteeth}{2}}}  &   
{\numavexts{213, 312}{\combB{\numteeth}{3}}} &
{\numavexts{213, 312}{\combB{\numteeth}{4}}} \\ \hline
2 & \observedval{2} & \observedval{2} & \observedval{2} \\ 
3 & \observedval{4} & \observedval{4} & \observedval{4} \\ 
4 & \observedval{8} & \extrapval{8} & \extrapval{8} \\ 
5 & \observedval{16} & \extrapval{16} & \extrapval{16} \\ 
6 & \observedval{32} & \extrapval{32} & \extrapval{32} \\ 
  \end{tabular}
\end{center}
\caption{{\numavexts{213, 312}{\combB{\numteeth}{\toothlen}}} for various values of {\toothlen} and {\numteeth}}
\end{figure}
}

\begin{thm}
\label{thm:numavexts213and312combB}
$\numavexts{213, 312}{\combB{\numteeth}{\toothlen}} = 2^{\numteeth - 1}.$
\end{thm}

\begin{proof}
In order for a linear extension of a type-{\typeB} comb to avoid $213$ and $312$, it must have no local minima; in other words, it must consist of a rising sequence followed by a falling sequence. Since any linear extension of a type-{\typeB} comb is constrained to have $\someelt \notin \{\toothlen, 2 \toothlen, \ldots, \numteeth \toothlen\}$ followed by $\someelt + 1$, all such elements $\someelt$ must be a part of the rising sequence. It remains to divide $\{\toothlen, 2 \toothlen, \ldots, (\numteeth - 1) \toothlen\}$ between the rising and falling sequences. (We disregard the element $\numteeth \toothlen$ here because it will be part of both sequences.) There are $2^{\numteeth-1}$ ways to split these $\numteeth - 1$ elements up into two groups, so it follows that $\numavexts{213, 312}{\combB{\numteeth}{\toothlen}} = 2^{\numteeth - 1}$.
\end{proof}

\subsection{Type-{\typeB} $213$- and $321$-avoidance}

\proventable{
\begin{figure}[H]
\begin{center}
\begin{tabular}{c | c | c | c}
  \hline
  {\numteeth} &  
{\numavexts{213, 321}{\combB{\numteeth}{2}}}  &   
{\numavexts{213, 321}{\combB{\numteeth}{3}}} &
{\numavexts{213, 321}{\combB{\numteeth}{4}}} \\ \hline
2 & \observedval{2} & \observedval{3} & \observedval{4} \\ 
3 & \observedval{3} & \observedval{5} & \extrapval{7} \\ 
4 & \observedval{4} & \extrapval{7} & \extrapval{10} \\ 
5 & \observedval{5} & \extrapval{9} & \extrapval{13} \\ 
6 & \observedval{6} & \extrapval{11} & \extrapval{16} \\ 
  \end{tabular}
\end{center}
\caption{{\numavexts{213, 321}{\combB{\numteeth}{\toothlen}}} for various values of {\toothlen} and {\numteeth}}
\end{figure}
}

\begin{thm}
\label{thm:numavexts213and321combB}
$\numavexts{213, 321}{\combB{\numteeth}{\toothlen}} = (\numteeth - 1)(\toothlen - 1) + 1$.
\end{thm}

\begin{proof}
In order for a linear extension of a type-{\typeB} comb to avoid $213$ and $321$, no decreasing two-element subsequence can be preceded or followed by an element larger than the first element of the subsequence. It follows that any decreasing two-element subsequence must have $\numteeth \toothlen$ as the first of its two elements and that there can be only one such decreasing subsequence; all the other elements must appear in order. $\numteeth \toothlen$ can either be the ultimate element or be followed by one of the elements not constrained by the minimal relations of {\combB{\numteeth}{\toothlen}} to appear before it, of which there are $(\numteeth - 1)(\toothlen - 1)$.
\end{proof}

\subsection{Type-{\typeB} $231$- and $312$-avoidance}

\proventable{
\begin{figure}[H]
\begin{center}
\begin{tabular}{c | c | c | c}
  \hline
  {\numteeth} &  
{\numavexts{231, 312}{\combB{\numteeth}{2}}}  &   
{\numavexts{231, 312}{\combB{\numteeth}{3}}} &
{\numavexts{231, 312}{\combB{\numteeth}{4}}} \\ \hline
2 & \observedval{2} & \observedval{2} & \observedval{2} \\ 
3 & \observedval{4} & \observedval{4} & \extrapval{4} \\ 
4 & \observedval{8} & \extrapval{8} & \extrapval{8} \\ 
5 & \observedval{16} & \extrapval{16} & \extrapval{16} \\ 
6 & \observedval{32} & \extrapval{32} & \extrapval{32} \\ 
  \end{tabular}
\end{center}
\caption{{\numavexts{231, 312}{\combB{\numteeth}{\toothlen}}} for various values of {\toothlen} and {\numteeth}}
\end{figure}
}

\begin{thm}
\label{thm:numavexts231and312combB}
$\numavexts{231, 312}{\combB{\numteeth}{\toothlen}} = 2^{\numteeth - 1}$.
\end{thm}

\begin{proof}
In order for a linear extension of a type-{\typeB} comb to avoid $231$ and $312$, any decreasing subsequence must be consecutive. It follows that these consecutive decreasing subsequences must appear in increasing order.
The minimal relations of {\combB{\numteeth}{\toothlen}} force elements {\someelt} not divisible by {\toothlen} to appear before $\someelt + 1$; thus, elements of the form $\someelt +1$ where {\someelt} is not divisible by {\toothlen} must all form their own length-$1$ consecutive decreasing subsequence. However, each element $\someelt < \numteeth \toothlen$ divisible by $\toothlen$ can appear either directly before or directly after $\someelt + 1$. Since there are $\numteeth - 1$ such elements {\someelt}, it follows that $\numavexts{231, 312}{\combB{\numteeth}{\toothlen}} = 2^{\numteeth - 1}$.
\end{proof}

\subsection{Type-{\typeB} $231$- and $321$-avoidance}

\proventable{
\begin{figure}[H]
\begin{center}
\begin{tabular}{c | c | c | c}
  \hline
  {\numteeth} &  
{\numavexts{231, 321}{\combB{\numteeth}{2}}} &   
{\numavexts{231, 321}{\combB{\numteeth}{3}}} &
{\numavexts{231, 321}{\combB{\numteeth}{4}}} \\ \hline
2 & \observedval{2} & \observedval{3} & \observedval{4} \\ 
3 & \observedval{4} & \observedval{9} & \extrapval{16} \\ 
4 & \observedval{8} & \extrapval{27} & \extrapval{64} \\ 
5 & \observedval{16} & \extrapval{81} & \extrapval{256} \\ 
6 & \observedval{32} & \extrapval{243} & \extrapval{1024} \\ 
  \end{tabular}
\end{center}
\caption{{\numavexts{231, 321}{\combB{\numteeth}{\toothlen}}} for various values of {\toothlen} and {\numteeth}}
\end{figure}
}

\begin{thm}
\label{thm:numavexts231and321combB}
$\numavexts{231, 321}{\combB{\numteeth}{\toothlen}} = {\toothlen}^{\numteeth - 1}$.
\end{thm}

\begin{proof}
Begin with the empty permutation and sequentially insert the largest {\toothlen} remaining elements of $\{1, \ldots, \toothlen \numteeth\}$. 

Each set of elements we insert will have to appear in order, since the elements form a tooth. We can insert $\{(\numteeth - 1)\toothlen + 1, \ldots, \numteeth \toothlen\}$ in exactly one way: $[(\numteeth - 1)\toothlen + 1, \ldots, \numteeth \toothlen]$.
Since in a $(231, 321)$-avoiding permutation no element can be preceded by two elements greater than itself, each following set of elements can be inserted in one of {\toothlen} ways: with $0, 1, \ldots, \toothlen - 2$ or $\toothlen - 1$ new elements succeeding the first element already in the permutation and all of the other elements preceding it. 
All {\toothlen} new elements cannot succeed the first element already in the permutation because of the spinal minimal relations of {\combB{\numteeth}{\toothlen}}.

It follows that $\numavexts{231, 321}{\comb{\numteeth}{\toothlen}} = {\toothlen}^{\numteeth - 1}$.
\end{proof}

\subsection{Type-{\typeB} $312$- and $321$-avoidance}

\proventable{
\begin{figure}[H]
\begin{center}
\begin{tabular}{c | c | c | c}
  \hline
  {\numteeth} &  
{\numavexts{312, 321}{\combB{\numteeth}{2}}}  &   
{\numavexts{312, 321}{\combB{\numteeth}{3}}} &
{\numavexts{312, 321}{\combB{\numteeth}{4}}} \\ \hline
2 & \observedval{3} & \observedval{4} & \observedval{5} \\ 
3 & \observedval{9} & \observedval{16} & \extrapval{25} \\ 
4 & \observedval{27} & \extrapval{64} & \extrapval{125} \\ 
5 & \observedval{81} & \extrapval{256} & \extrapval{625} \\ 
6 & \observedval{243} & \extrapval{1024} & \extrapval{3125} \\ 
  \end{tabular}
\end{center}
\caption{{\numavexts{312, 321}{\combB{\numteeth}{\toothlen}}} for various values of {\toothlen} and {\numteeth}}
\end{figure}
}

\begin{thm}
\label{thm:numavexts312and321combB}
$\numavexts{312, 321}{\combB{\numteeth}{\toothlen}} = {(\toothlen + 1)}^{\numteeth - 1}$.
\end{thm}

\begin{proof}
Let's show this by beginning with the empty permutation and sequentially inserting the smallest {\toothlen} remaining elements of $\{1, \ldots, \toothlen \numteeth\}$. 

Each set of elements we insert will have to appear in order, since the elements form a tooth. We can insert $\{1, \ldots, \toothlen\}$ in exactly one way: $[1, \ldots, \toothlen]$.
Since in a $(312, 321)$-avoiding permutation no element can be succeeded by two elements smaller than itself, each following set of elements can be inserted in one of $\toothlen + 1$ ways; with $0, 1, \ldots, \toothlen - 1$ or $\toothlen$ new elements preceding the last element already in the permutation, and all of the other elements succeeding it. 

It follows that $\numavexts{312, 321}{\combB{\numteeth}{\toothlen}} = {(\toothlen + 1)}^{\numteeth - 1}$.
\end{proof}

\section{Conclusion}

Even within pattern-avoiding linear extensions of combs for patterns of length $3$, there is more work to be done:
the enumeration of $231$-avoiding extensions of type-{\typeA} combs and the enumeration of $312$-avoiding extensions of type-{\typeA} combs remain open problems.

Beyond that, there are a lot of exciting threads to be followed. Longer patterns can be considered, and other --- ideally more general --- classes of posets can be explored.

\section{Acknowledgements}

I would like to thank Professor Richard Stanley for his support, advice and patience. 
His Combinatorial Analysis class (MIT class number 18.314) is what sparked my interest in permutation patterns, and he has continued to mentor me since.
I would also like to than Professor Leonid Reyzin for his guidance on writing papers, 
Nina Shteingold for her illustration, 
and Vladimir Shander as well as the anonymous JoC reviewers for their insightful feedback.

\section{Contact Information}

Please feel free to contact the author with any questions or comments via email at sonka89@mit.edu.

\bibliographystyle{plain}
\bibliography{21_bibliography}

\end{document}